\documentclass[12pt]{article}
\usepackage{booktabs} 
\usepackage{epstopdf}
\usepackage{amsmath}
\usepackage[standard]{ntheorem}
\usepackage{hyperref}

\author{Fedor Duzhin, Nanyang Technological University}

\title{Learning in teams: peer evaluation for fair assessment of individual contributions.}

\begin{document}

\maketitle

\begin{abstract}
The "free rider" problem has long plagued pedagogies based on collaborative learning. The most common solution to the free rider problem is peer evaluation. As well other existing methods of peer evaluation include self-evaluation --- and hence are prone to grade inflation or, as we show here, are inaccurate in that they do not fairly reward the most hard working student. Another common concern with existing methods of peer evaluation is that students often do not have the necessary skills to evaluate the work of their peers objectively.

In this paper, we introduce a new mechanism for peer evaluation that does not rely on self-evaluation, and yet remains accurate, i.e., if all students are completely truthful in their evaluations, then the output of our mechanism becomes an objective truth. At the same time, our mechanism integrates the instructor's judgment with respect to the credibility of students' evaluations. For example, the instructor gives scores to students for writing credible reviews and, in turn, subsequent students' evaluations are weighted according to these instructor scores.
\end{abstract}


\section{Introduction}

A vast body of literature exists on methods of assessment in tertiary education --- see, for example, \cite{palomba1999assessment}. In practice, however, written final exams prevail, even though most students will never take an exam in their life after graduation and therefore exam grades are hardly able to capture the true potential of a student to thrive in a complex work environment.

Even though most people never take formal exams after leaving school, working in teams and writing reports are typical job tasks of modern homo sapiens. Team work and report writing are taught at universities, but grading every individual student fairly based on a team's report is a challenge. For instance, if all the team members get the same grade, then a free-rider problem may occur (see \cite{leuthold1993free}, \cite{joyce1999free}, \cite{brooks2003free}).

The most obvious solution of the free-rider problem is peer evaluation (see \cite{conway1993peer}). The simplest and yet popular approach to peer evaluation is letting each student grade contribution of each of the team members in absolute terms, i.e., out of 10, a 100, as A, B, C or in a similar way (see, for example,
\cite{holter1994team}, \cite{malik1998exploration}, or \cite{morse2014international}). According to our experience, peer evaluation in absolute terms results to almost all students giving maximal scores to each other just not to jeopardize their friends' final grade. Thavikulwat and Chang criticize the whole idea of peer evaluation in \cite{thavikulwat2014pick} and propose to replace it with a different procedure based on students choosing their preferred group size. 

Note that while the free rider problem may plague a variety of collaborative learning activities, such as team-based learning (see \cite{stein2016student}), the main scenario that we have in mind is a team of students collaborating on a well-defined list of tasks. The group size then should equal, roughly, the number of individual tasks within the project and therefore varying the group size is not an option.

There exist sophisticated peer evaluation systems. The system described in \cite{brooks2003free} is based on each student allocating a certain number of points between their teammates. Kauffman et al. introduce in \cite{kaufman2000accounting} a mixed system where students give each other ratings from a list of nine terms such as "excellent", "very good" etc. but these ratings are then converted into a numeric value by dividing everyone rating by the team's average.

A rigorous mathematical theory of peer evaluation is outlined in \cite{carvalho2012sharing}. However, the theoretical work \cite{carvalho2012sharing} is very broad and aimed at mathematicians --- experts in game theory. In this paper, we are going to narrow the scope of the theory and to simplify it so it becomes accessible by educators and education researchers.

A system of peer evaluation is a procedure of calculating the "true'' (at least, as it is perceived by team members) contribution of each of the team members into the common task based on mutual evaluations reported by team members. A system of peer evaluation may or may not have certain desired qualities. One such quality is resistance to grade inflation. A system of peer evaluation that includes self-evaluation as part of the process is prone to grade inflation because students have incentive to overestimate their own contribution. The second desired quality is \emph{accuracy}. A system of peer evaluation is accurate if it outputs true contributions of each of the team members whenever they all report the truth. As we show below, there exist systems of peer evaluation that are accurate but prone to grade inflation and systems that do not rely on self-evaluation but are not accurate.  The third quality that may also be desired of a peer evaluation system is integration of the instructor's judgment into the system. The reason for integrating the instructor's judgment into the system is that students either may not have the expertise to evaluate their peers or may not have incentive to do it fairly. The instructor will then serve as moderator.

In this paper, we develop a system of peer evaluation that is accurate, does not rely on self-evaluation, and integrates the instructor's judgment. 

\section{Mathematical Theory}

We assume that $n$ students collaborate on a common goal of completing a set of well-defined tasks and there exist the objective truth --- the share of total work that each student has accomplished. If the true contribution/share of the $i$\,th student is $t_i$, then the objective truth is the vector
\begin{multline*}
t=(t_1,t_2\dots,t_n), 
\mbox{ where } t_1\ge 0,t_2\ge 0,\dots,t_n\ge 0\\
\mbox{ and }  t_1+t_2+\cdots+t_n=1
\end{multline*}
The vector $t$ is known to students but can't be observed by the course instructor directly and the system of peer evaluation should motivate students to reveal the truth to the instructor. What students report to the instructor is a matrix $A$ of evaluations of each student by each student. Let entries of this matrix be denoted $a_{ij}$ -- evaluation of student $i$ by student $j$. We assume that $a_{ij}\ge 0$ (even though systems of peer evaluation with negative scores exist, one can convert such a system to a system with non-negative scores simply by taking the exponential function of each score).

For each $i$, the vector $A_{i*}=(a_{i1},a_{i2},\dots,a_{in})$ is the vector of evaluations received by student $i$ and, for each $j$, the vector $A_{*j}=(a_{1j},a_{2j},\cdots,a_{nj} )$ is the vector of evaluations reported by student $j$.

\subsection{Mechanism}

A \emph{mechanism} (term adopted from \cite{carvalho2012sharing}) is an algorithm of calculating the vector of perceived students' contributions / shares of workload
\begin{multline*}
s=(s_1,s_2\dots,s_n), 
\mbox{ where } s_1\ge 0,s_2\ge 0,\dots,s_n\ge 0\\
\mbox{ and }  s_1+s_2+\cdots+s_n=1
\end{multline*}
from the matrix $A=(a_{ij} )_{1\le i\le n, 1\le j\le n}$. Note that the output of the mechanism, i.e., the vector $s$ of perceived contributions may or may not be equal to the vector $t$ of true contributions.

A mechanism may or may not rely on self-evaluation. If students are not required to report a self-evaluation, we let $a_{ii}=0$ for all $i$. A mechanism that relies on self-evaluation will be prone to \emph{grade inflation} --- a student will be tempted to give an unfairly high evaluation to himself. 

A mechanism is \emph{accurate} if outputs true contributions of each team member whenever all students are truthful. In other words, if, for each $j$, the vector 
$A_{*j}=(a_{1j},a_{2j},\cdots,a_{nj} )$
reported by student $j$ is proportional to the vector of true contributions $t=(t_1,t_2,\dots,t_n)$, then we must get $s=t$.

If the course instructor's purpose is to fairly evaluate each student's individual contribution to the team effort, the mechanism should be accurate and should not rely on self-evaluation. However, as we show below,  mechanisms that are widely used in practice, usually have one of these two qualities, but not both. Currently, mechanisms used in practice are mostly variations of one of the following two.

\subsubsection*{Pie-to-all mechanism}

The \emph{pie-to-all} mechanism works as follows.
Each student gets one pie and then distributes her pie among all team members, including herself, in proportion to their contribution to the team effort. The final perceived contribution of a team member is the average piece of pie that he received from all the team members. It means that we have
\begin{equation}\label{equation: pie-to-all}
a_{ij}\ge 0,\quad \sum_{i=1}^{n}a_{ij}=1,\quad
s_i=\frac{1}{n}\sum_{j=1}^{n}a_{ij}.
\end{equation}
It is easy to show that pie-to-all is an accurate mechanism. However, it relies on self-evaluation and hence is prone to grade-inflation. Clearly, the best strategy to student $j$ is to report $a_{jj}=1$ and $a_{ij}=0$ for $i\neq j$.

Still, the pie-to-all mechanism has been used in practice --- see, for example,  \cite{kaufman2000accounting}.

\subsubsection*{Pie-to-others mechanism}

The \emph{pie-to-others} mechanism works as follows.
Each student gets one pie and then distributes her pie among all her teammates, not including herself, in proportion to their contribution to the team effort.  The final perceived contribution of a team member is the average piece of pie that he received from all the team members. It means that we have
\begin{equation}\label{equation: pie-to-others}
a_{ij}\ge 0,\quad a_{ii}=0,\quad \sum_{i=1}^{n}a_{ij}=1,\quad
s_i=\frac{1}{n}\sum_{j=1}^{n}a_{ij}.
\end{equation}
The only difference with the pie-to-all mechanism \eqref{equation: pie-to-all} is the absence of self-evaluation, which is expressed by $a_{ii}=0$.

The pie-to-others is a very popular mechanism, probably the most popular one. Its clear advantage is that it does not rely on self-evaluation. However, a huge issue with the pie-to-others mechanism is that it is not accurate as the following example shows.

\begin{example}
Suppose that we have a team of three students and the vector of true contributions is
\[
t=\left(\frac{1}{2}, \frac{1}{4},\frac{1}{4}\right).
\] 
The matrix of peer-evaluations and the the vector of perceived contributions are
\[
A=
\begin{bmatrix}
0 & 2/3 & 2/3 \\
1/2 & 0 & 1/3 \\
1/2 & 1/3 & 0
\end{bmatrix},\quad 
s=
\begin{bmatrix}
4/9 \\ 5/18 \\ 5/18   
\end{bmatrix} \neq 
\begin{bmatrix}
1/2 \\ 1/4 \\ 1/4  
\end{bmatrix}
\]
\end{example}

In our experience, real students, at least those whose major is mathematics, understand very well that the pie-to-others mechanism is not accurate. Students will not be happy if a large portion of their grade comes from this mechanism.

However, the inaccuracy of  pie-to-others is not the only problem of this mechanism. Well,  pie-to-others still gives the highest score to the most hard-working student in a team, even though that score may not accurately reflect the true contribution. Let's demonstrate the other issue by an example. 

\begin{example}
Consider a hypothetical team of three students whose vector of true contributions is 
\[
t=\left(\frac{1}{2}, \frac{1}{2},0\right).
\]
The matrix of peer evaluations and the the vector of perceived contributions are
\[
A=
\begin{bmatrix}
0 & 1 & 1/2  \\
1 & 0 & 1/2 \\
0 & 0 & 0
\end{bmatrix},\quad 
s=
\begin{bmatrix}
1/2 \\ 1/2 \\ 0   
\end{bmatrix}=t,
\]
which means that now all the students are fairly rewarded. 
\end{example}

Thus it is profitable for strongest students to just do all the work by themselves without letting their teammates do anything. It's then (and only then) that they will be fairly rewarded for their hard work. This behavior is very common in practice and a serious weakness of the pie-to-others mechanism is that it incentivizes such behavior.

\section{Results}

In this section, we outline a new mechanism that is accurate but does not rely on self-evaluation.

\subsection{Auxiliary matrix}

Let $a_{ij}$ be raw evaluations of student $i$ by student $j$ for $i,j\in \{1,2,\dots,n\}$. We are going to construct an \emph{auxiliary matrix} $B$ whose entries $b_{ij}$ show relative contributions of students $i$ and $j$, the ratio of $i$'s contribution to $j$'s contribution according to other team members. If $k$ is any student other than $i$ and $j$, then 
\begin{equation}\label{equation: relative contribution}
\frac{a_{ik}}{a_{ik}+a_{jk}} 
\end{equation}
is the share of $i$'s contribution in the combined $i$'s and $j$'s contribution according to $k$. Therefore,
\[
\frac{1}{n-2}\sum_{k\neq i, j}\frac{a_{ik}}{a_{ik}+a_{jk}}
\]
is the average share of $i$'s contribution in the combined $i$'s and $j$'s contribution according to their teammates. Thus,
\begin{equation}\label{equation not-weighted b_i_j}
b_{ij}=\frac{\sum_{k\neq i, j}\frac{a_{ik}}{a_{ik}+a_{jk}}}
{\sum_{k\neq i, j}\frac{a_{jk}}{a_{ik}+a_{jk}}}
\end{equation}
is the average ratio of $i$'s contribution to $j$'s contribution according to their teammates. 

Note that \eqref{equation: relative contribution} does not change if we multiply all evaluations reported by a student by a constant. It means that it is not necessary to normalize the matrix of peer-evaluations $A$ to compute the auxiliary matrix $B$.

\begin{example}\label{example 5 calculation without weights}
Consider a team of 3 students with true contributions 
$t=(0.2,0.3,0.5)$. The matrix of peer-evaluations and the auxiliary matrix are
\[
A=
\begin{bmatrix}
0 & 2 & 2 \\
3 & 0 & 3 \\
5 & 5 & 0 
\end{bmatrix},\quad 
B=
\begin{bmatrix}
1  & 2/3 & 2/5 \\
3/2 & 1  & 3/5 \\
5/2 & 5/3 & 1
\end{bmatrix}
\]
Note that each column of the matrix $B$ is proportional to $t$.
\end{example}

\subsection{Instructor's judgment}

It is not realistic to expect that actual students would be completely truthful and absolutely precise in their evaluations. In our experience, students who contribute least of all usually tend not to put too much thought into evaluations that they submit and often simply give equal scores to everyone. It is therefore important to have some procedure of discrediting evaluations that are not trustworthy.

In our mechanism, students not only give numeric evaluations to each other, but also provide justifications, i.e., write short reports on each of their teammates. The instructor will then read these reports and give students grades for writing the good trust-worthy reports. Let $w_i$ be the grade given to student $i$ by the instructor for writing reports.

We will now modify the construction of the auxiliary matrix $B$ in order to integrate trustworthiness of students' reports into it. Specifically, \eqref{equation not-weighted b_i_j} is replaced with 
\begin{equation}\label{equation weighted b_i_j}
b_{ij}=\frac{\sum_{k\neq i, j}\frac{w_k \cdot a_{ik}}{a_{ik}+a_{jk}}}
{\sum_{k\neq i, j}\frac{w_k \cdot a_{jk}}{a_{ik}+a_{jk}}},
\end{equation}
i.e., the average is replaced with weighted average.

\begin{example}\label{example 6 calculation with weights} 
Consider four students and true contributions 
\[
t=(0.1,0.2,0.3,0.4).
\]
Suppose also that the vector of grades for writing reports is $w=(4,0,1,3)$. Probably, student 1 wrote nice trustworthy reports, student 2 did not write anything, student 3 wrote something short and not very meaningful, and student 4 wrote meaningful reports but missed some points. At the same time, students 1 and 4 are truthful in their evaluations, student 2 gave same scores to everyone and student 3 was not completely truthful so that the matrix of raw evaluations is
\[
A=
\begin{bmatrix}
0 & 1 & 11 & 1 \\
2 & 0 & 19 & 2 \\
3 & 1 &  0 & 3 \\
4 & 1 & 39 & 0
\end{bmatrix}
\]
Then the auxiliary matrix is 
\[
B=
\begin{bmatrix}
1       &  0.518987 &  0.333333  &  0.282051 \\
1.92683 &  1        &  0.666667  &  0.497418 \\
3       &  1.5      &  1         &  0.75 \\
3.54545 &  2.01038  &  1.33333   &  1
\end{bmatrix}
\]
For instance,
\[
b_{12}=\frac{\frac{1\cdot a_{13}}{a_{13}+a_{23}}+\frac{3\cdot a_{14}}{a_{14}+a_{24}}}{\frac{1\cdot a_{23}}{a_{13}+a_{23}}+\frac{3\cdot a_{24}}{a_{14}+a_{24}}}
=\frac{\frac{11}{11+19}+\frac{3\cdot 1}{1+2}}
{\frac{19}{11+19}+\frac{3\cdot 2}{1+2}}=0.518987
\]
\end{example}

\subsection{Missing values}

In practice, some students fail to submit evaluations of their teammates, i.e., a real matrix $A$ may have missing values. In this case, the instructor will give a zero score for reports to students who did not write any reports and impute their missing numeric evaluations with equal scores. Such non-existing evaluations will be then automatically discarded.

\subsection{Main mechanism}

To calculate "true" individual contributions of each team member from a matrix $A$ of peer evaluations, we let the instructor grade reports written by each student on their teammates. Let $w_k$ be the grade given to student $k$ by the instructor, we  calculate the auxiliary matrix $B$ by \eqref{equation weighted b_i_j}.
Then the entry $b_{ij}$ of the auxiliary matrix is the relative contribution of student $i$ to student $j$ according to their teammates. We also have $b_{ii}=1$ for all $i$ and $b_{ij}=\frac{1}{b_{ji}}$.

Note that if everyone were completely truthful, then each column of the matrix $B$ would be proportional to the true contributions of the team members. In practice, however, columns of the matrix $B$ may not be proportional to each other due to difference in students' opinions. Another practical issue is that there may be students who don't contribute at all. If $j$ is such a student, i.e., $t_j=0$, then we would have
\[
b_{ij}=
\begin{cases}
1, & \mbox{ if } i=j,\\
\infty, &\mbox{ if } t_i>0,\\
\mbox{undefined}, & \mbox{ if } i\neq j \mbox{ and }t_i=0
\end{cases}
\]
The final scores are computed from the matrix $B$ by
\begin{equation}\label{equation our mechanism}
s=
\mathrm{Average}_{j}
\frac{B_{*j}}{\sum_{i=1}^{n} b_{ij}},
\end{equation}
where the average is taken over all columns $j$ that do not have infinite or missing values. Note that we normalize each column $B_{*j}$ dividing it by the sum of its entries, to ensure that sum of entries of the vector $s$ is $1$.

\begin{theorem}
Suppose that at least two students made positive contribution to the team's work and that there are at least three students who evaluated each of their teammates and got positive grades from the instructor for reports that they wrote.
Then the mechanism given by \eqref{equation our mechanism} is accurate. 
\end{theorem}

\begin{proof}
The assumptions are needed to ensure that the matrix $B$ is well-defined and has at least two columns with finite entries.
The rest is straightforward --- to show that the mechanism is accurate, we need to prove that its output is the objective truth assuming that all evaluations are truthful. If all the students report the objective truth, then we have
\[
a_{ij}=\frac{t_i}{1-t_j}
\] 
From \eqref{equation weighted b_i_j}, we get
\[
b_{ij}=
\frac{\sum_{k\neq i, j}\frac{w_k \cdot \frac{t_i}{1-t_k}}{\frac{t_i}{1-t_k}+\frac{t_j}{1-t_k}}}
{\sum_{k\neq i, j}\frac{w_k \frac{t_j}{1-t_k}}
{\frac{t_i}{1-t_k}+\frac{t_j}{1-t_k}}}
=
\frac{\sum_{k\neq i, j}\frac{w_k t_i}{t_i+t_j}}
{\sum_{k\neq i, j}\frac{w_k t_j}
{t_i+t_j}}
=
\frac{t_i\sum_{k\neq i, j}\frac{w_k }{t_i+t_j}}
{t_j \sum_{k\neq i, j}\frac{w_k}{t_i+t_j}}=\frac{t_i}{t_j},
\]
which completes the proof.
\end{proof}

\subsection{Evaluating truthfulness}

Although our mechanism is accurate and does not rely on self-evaluation, there is still a practical consideration that it does not address. Sometimes, two friends may cheat the system by giving unfairly high evaluations to each other. If they write convincing reports, the course instructor will overlook the deceit.

The following method was proposed in \cite{carvalho2012sharing} to discourage such behavior. We actually require students to report self-evaluations, i.e., we have $a_{ii}>0$. Even though self-evaluations are discarded when the mechanism output is computed according to \eqref{equation our mechanism}, all evaluations reported by a student, including the self-evaluation, are used to measure to which extend evaluations $A_{*j}$ reported by student $j$ deviate from the final scores $s$.

Let us introduce the notation
\[
c_{ij}=\frac{a_{ij}}{\sum_{i=1}^{n}a_{ij}}
\]
for contribution of student $i$ according to student $j$. Then
\[
\left|\frac{c_{ij}-t_j}{t_j}\right|
\]
is the relative error of evaluation of student $i$ by student $j$ and
\begin{equation}\label{equation mean evaluation error}
E_j=\frac{1}{n}\sum_{i=1}^{n}\left|\frac{c_{ij}-t_j}{t_j}\right|
\end{equation}
is the average relative error of student $j$'s evaluations.

\subsection{Summary of our mechanism}

In practice, the main score computed by the mechanism according to \eqref{equation our mechanism} constitute 90\% of the final score, with 5\% given for writing reports and 5\% given for consistency of evaluations submitted by a student with scores calculated by the main mechanism. Thus the final score given to student $i$ is
\begin{equation}\label{equation final score}
0.9s_i+0.05w_i+0.05(1-\min(1, E_i)),
\end{equation}
where $s_i$ is the output of the main mechanism computed according to \eqref{equation our mechanism}, $w_i$ is the score given by the instructor to student $i$ for reports that $i$ wrote, and $E_i$ is the average relative error of $i$\,th evaluations found according to \eqref{equation mean evaluation error}.

Note that the average of scores given by \eqref{equation final score} is usually smaller than $1$ and equals $1$ if and only if all evaluations submitted by students are completely truthful. Another remark is that the weights $0.9$, $0.05$, and $0.05$ for the output of the main mechanism, reports, and consistency of evaluations with output of the main mechanism are not justified by data or theory and finding "right" weights may be a subject of future research.

\section{Discussion}

Our mechanism is accurate. However, even though it does not rely on self-evaluation, it may be possible to manipulate it. 

\begin{example}
Suppose that we have a team of three students and the vector of true contributions is
\[
t=\left(\frac{1}{2}, \frac{1}{4},\frac{1}{4}\right).
\] 
Suppose that the third student decided to manipulate the mechanism and reported incorrect evaluations by claiming that the first and the second student contributed equally. Assume also that the course instructor did not see it and graded all reports as being equally trust-worthy. The matrix of peer-evaluations and the auxiliary matrix are
\[
A=
\begin{bmatrix}
0 & 2/3 & 1/2 \\
1/2 & 0 & 1/2 \\
1/2 & 1/3 & 0
\end{bmatrix},\quad 
B=
\begin{bmatrix}
1 & 1 & 2 \\
1 & 1 & 1 \\
1/2 & 1 & 1
\end{bmatrix}
\]
The vector of percieved contributions is then
\[
s=\frac{1}{3}
\begin{bmatrix}
2/5 \\
2/5 \\
1/5
\end{bmatrix} + 
\frac{1}{3}
\begin{bmatrix}
1/3 \\
1/3 \\
1/3
\end{bmatrix} +
\frac{1}{3}
\begin{bmatrix}
1/2 \\
1/4 \\
1/4
\end{bmatrix}
\]
Then the perceived contribution of the third student is $\frac{47}{180} > \frac{1}{4}$. It means that by changing his evaluations of other students, he changed his own score. 
\end{example}

Note that although manipulating our mechanism is not as straightforward as manipulating the pie-to-all mechanism, it is still possible. 

A mechanism is called \emph{incentive compatible} if lying does not improve one's own score given that others tell the truth. In other words, a mechanism is incentive compatible if the collective truth-telling is a Nash equilibrium. A mechanism is \emph{not manipulable} if changing evaluations of other students one reports does not change one's own score. Note that a mechanism that is not manipulable is automatically incentive compatible, but not vice versa.

Our mechanism is manipulable, but we believe that it will become incentive compatible if a small portion of the final score is given for consistency of scores reported by a student with the team's perceived contributions. However, determining that small portion (in practice, we use $10\%$, but we don't have justification for that number) is still an open problem. We don't know if there exists an accurate mechanism that is not manipulable. We believe that it does not, but proving it is another open problem.

\bibliographystyle{plain}
\bibliography{references}

\end{document}